\documentclass[11pt, oneside]{article}   	
\usepackage{geometry}                		
\usepackage{amsthm}  
\usepackage{url}        		
\usepackage{graphicx}				
\usepackage{amssymb}
\usepackage{amsmath}
\newtheorem{theorem}{Theorem}
\newtheorem{lemma}{Lemma}
\newtheorem{claim}{Claim}
\newtheorem{remark}{Remark}
\newcommand*{\QED}{\hfill\ensuremath{\square}}

\title{Outer Billiards and Tilings of \\ the Hyperbolic Plane}
\author{\textsc{Filiz Dogru} - Grand Valley State University\\\textsc{Emily M. Fischer} - Harvey Mudd College \\ \textsc{Cristian Mihai Munteanu} - Jacobs University Bremen}
\date{}

\begin{document}
\maketitle
\begin{abstract}
In this paper we present new results regarding the periodicity of outer billiards in the hyperbolic plane around polygonal tables which are tiles in regular two-piece tilings of the hyperbolic plane.
\end{abstract}
\section{Introduction}

\indent \indent Outer billiards is a simple dynamical system that was  introduced  by B. H. Neumann in 1950s in \cite{neumann}. In the 1970s, J. Moser popularized outer billiards as a toy model for planetary motion as a means of finding possible unbounded orbits \cite{moser1,moser}. Since then, many mathematicians have asked and answered questions about outer billiards systems in various geometries. For example, in 2004 C. Culter proved the existence of periodic orbits for polygonal tables in the Euclidean Plane (the proof is presented by S. Tabachnikov in  \cite{culter}).  R. Schwarz answered, in the affirmative, Moser's question about the existence of unbounded orbits for certain polygons in \cite{schwartz1,schwartz}. 

The main motivation for this paper is a result of Vivaldi and Shaidenko [8] that in the Euclidean case, outer billiards associated to quasi-rational polygons have  all orbits  bounded, see also \cite{Ko,GS}. As a consequence, all orbits about a lattice polygon in the Euclidean plane are periodic.
We continue the work of Dogru and Tabachnikov in \cite{dogtab} who studied the relationship between one-tile regular tilings of the hyperbolic plane and the outer billiards system. 

 For a detailed account of hyperbolic geometry and the hyperbolic plane, we direct the reader to  \cite{hyperbolic}, and for a survey of outer billiards, see \cite{dogtab1,Tab}.

\section{Definitions}
\indent \indent The outer billiard map associated to a convex polygonal table $P$ in the hyperbolic plane is defined as follows. For a point $x\in\mathbb{H}^2\setminus P$, there are two lines that pass through $x$ and are tangent to the table $P$. By convention, we consider the tangent line for which $P$ is on the left, from the point of view of $x$. Then we reflect $x$ about the tangency (support) point to get $T(x)$ (See Figure \ref{definition}).  The map is well-defined whenever the tangency point is unique and so we are able to define the map $T$ on the entire hyperbolic plane except for the clockwise continuations of the sides of $P$ (see Figure \ref{definition}) and their preimages under $T$.  An immediate consequence of the definition is that $T$ is a piecewise isometry. 

Likewise, the inverse map $T^{-1}$ is not defined on the counterclockwise continuations of the sides of $P$.
We define the \textit{web} associated to $P$ to be the union of all preimages under $T$ of the clockwise continuation of the sides and of all preimages under $T^{-1}$ of the counterclockwise continuation of the sides. For each connected component of the complement of the web, the restriction of the map $T^n$ to that component is defined by a single isometry of the hyperbolic plane for every $n\in {\mathbb Z}$.  That means that each connected component of the complement of the web  maps as a whole under the iterations of $T$.
\begin{figure}
\centering
\includegraphics[scale=0.6]{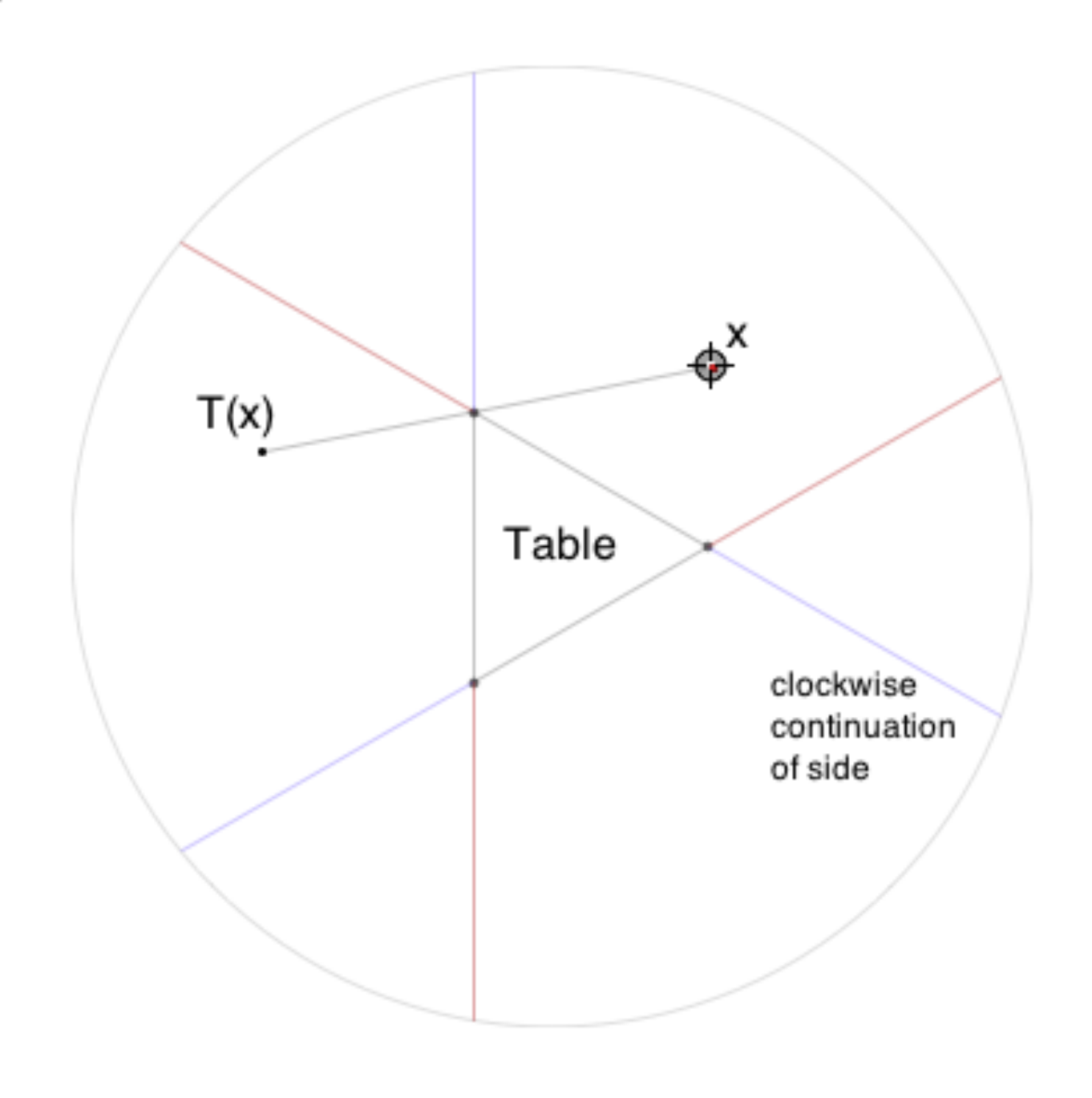}
\caption{Outer billiard map in Klein model}
\label{definition}
\end{figure}

\indent Another feature of the billiards map $T$ is that it extends continuously to a continuous circle map $t:S^1\to S^1$ at infinity. The map $t$ is defined using the same reflecting procedure. In this case the uniqueness of the support point is not needed, since the distance between our initial point and the support point is infinite no matter the choice and hence the map $t$ is well-defined for every point at infinity. Since $t$ is a circle map, it has a well defined Poincar\'{e} rotation number $\rho(t)$, and we will prove in section 3 that $\rho(t)$ encodes information about the combinatorial dynamics of the outer billiards.


\section{Outer Billiards on Tilings}
\indent \indent We are studying the hyperbolic outer billiards map associated with a polygonal table that is part of a two piece regular tiling of the hyperbolic plane. These tilings use two polygonal pieces, a regular $M$-gon and a regular $N$-gon that meet four in each vertex (See Figure \ref{mntiling}). We describe the combinatorial dynamics for outer billiards around one of the $M$-gons. We note that the web associated to such a map will fall exactly on the grid lines of the tiling. This is because the reflection around a vertex of the table tile is just a rotation by $180^{\circ}$ around vertices in the tiling. It follows that each tile maps as a whole under iterations of $T$.
\begin{figure}
\centering
\includegraphics[scale=0.3]{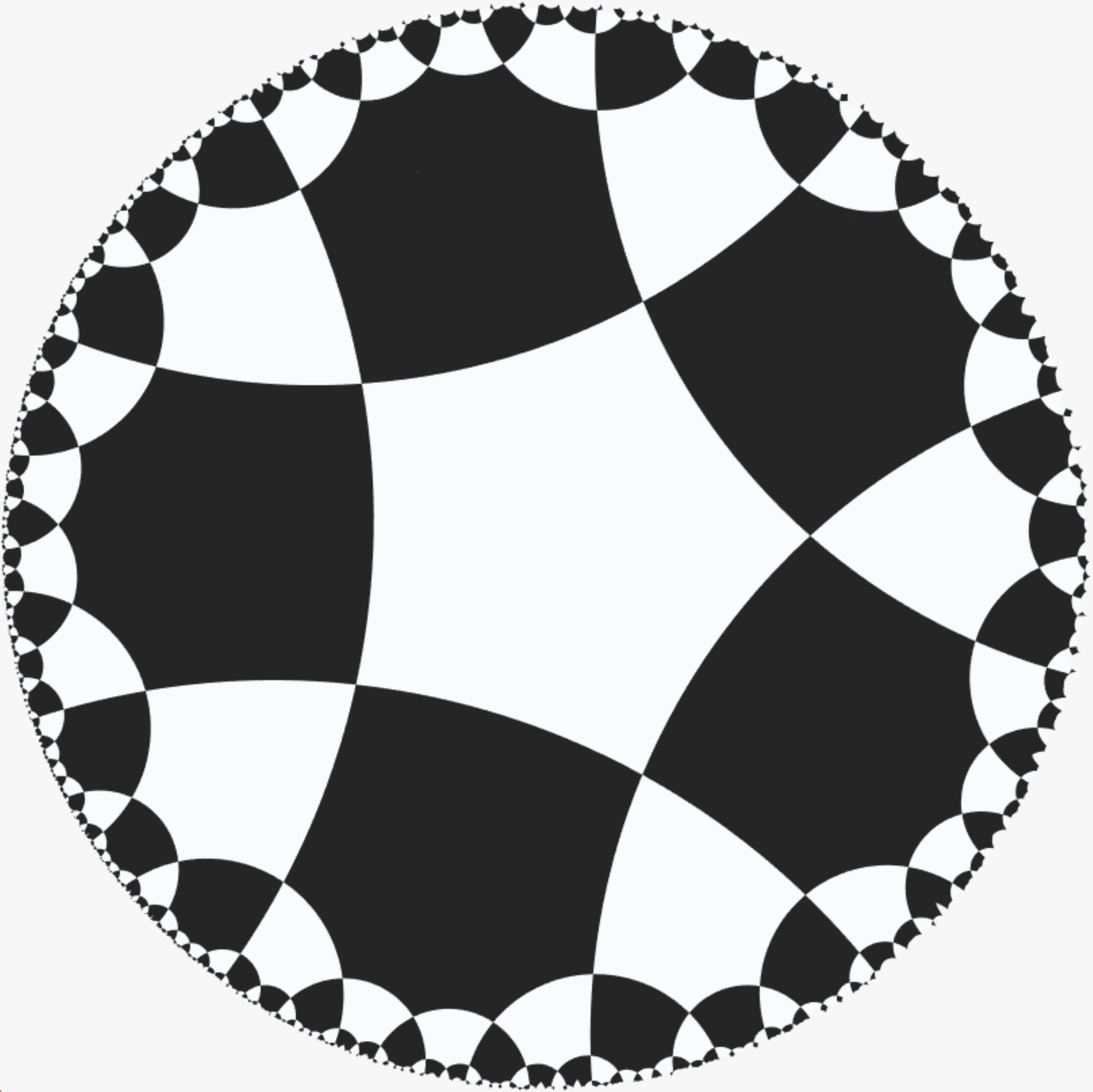}
\caption{Example of $(M,N)$-tiling for $(M,N)=(6,7)$}
\label{mntiling}
\end{figure}
\subsection{Previous Results}
\indent \indent Previous results describing outer billiards of tiles in the hyperbolic plane are obtained  in \cite{dogtab}. In this paper, the authors have proved that every orbit of the outer billiard map around a right-angled regular $n$-gon, for $n \geq 5$, is periodic. Any right-angled regular $n$-gon generates a tiling of the hyperbolic plane entirely consisting of $n$-gons. The theorems proven in the next sections have the same flavor as Theorem 4 in the above mentioned paper.

\indent Define the \textit{rank} of a tile as the minimum number of sides that one has to cross, when starting inside the table, to get to the given tile. This means that tiles that have one common side with the table have rank 1, and tiles that have a common side with a tile of rank 1 have rank 2, and so on. 
\begin{theorem}
(Dogru-Tabachnikov \cite{dogtab})
For a tiling of regular $n$-gons meeting in 4, $n\geq 5$, the dual billiard map $T$ preserves the rank of a tile, and every orbit of $T$ is periodic. The set of rank $k$ tiles consists of 
\[q_k=n\frac{\lambda_1^k-\lambda_2^k}{\lambda_1-\lambda_2}\]
elements, where
\[\lambda_{1,2}=\frac{n-2\pm\sqrt{n(n-4)}}{2}\]
are the roots of the equation $\lambda^2-(n-2)\lambda+1=0$. The action of $T$ on the set of rank $k$ tiles is a transitive cyclic permutation $i\mapsto i+p_k$ where 
\[p_k=\frac{\lambda_1^{k-1}-\lambda_2^{k-1}}{\lambda_1-\lambda_2}+\frac{\lambda_1^k-\lambda_2^k}{\lambda_1-\lambda_2}.\]
The rotation number of the dual billiard map at infinity is given by the formula 
\[\rho(t)=\lim_{k\to\infty}\frac{p_k}{q_k}=\frac{n-\sqrt{n(n-4)}}{2n}.\]
\end{theorem} 
The proof of this theorem uses geometric arguments for the periodicity of orbits and recurrence formulas for computing the number of tiles in each rank and the rotation number of $t$ (see \cite{dogtab} for details). The authors make an important remark that the representation of $\lambda_1$ (and so the rotation number of the map at infinity) as a continued fraction encodes the dynamics of the tiles under the billiard map $T$. We will deduce  similar results for two-piece tilings.
\subsection{New Results}
\indent \indent Our results extend Theorem 1 to two-piece regular tilings of the hyperbolic plane. We will denote a tiling of regular $M$-gons and regular $N$-gons as an $(M,N)$-tiling, and we will always consider the table to be an $M$-gon. Such an $(M,N)$-tiling exists if $\frac{1}{M}+\frac{1}{N}<\frac{1}{2}$. As mentioned earlier, these tilings have four shapes meeting at each vertex, two $M$-gons and two $N$-gons.
\subsubsection{Triangles and $N$-gons}\label{sec:3-N}
\indent \indent Most of the geometric arguments used here are analogous to those used by Dogru and Tabachnikov. Our counting arguments are different, although they are also based on recurrence relations. 

Let us introduce a more general notation for rank in order to avoid cumbersome indexing. Observe that the layer of tiles of rank $k$ includes tiles of the same type (all $M$-gons or all $N$-gons) and as rank changes by one, that shape changes. So triangles always have even rank and $N$-gons always have odd rank. We will say that a rank $2k-1$ tile is a rank $k$ $N$-gon and a rank $2k$ tile is a rank $k$ triangle. The rest of this section is dedicated to describing the dynamics of the billiard map $T$ in the $(3,N)$-tilings through the proof of the following theorem:
\begin{theorem}\label{3n}
For a $(3,N)$-tiling, $N\geq 7$, the outer billiard map $T$ preserves the rank of a tile and every orbit of $T$ is periodic. The set of rank $k$ $N$-gons consists of 
\[q_k=\frac{1}{\sqrt{N-6}}(\Phi_1^{2k-3}+\Phi_2^{2k-3})+\Phi_1^{2k-2}+\Phi_2^{2k-2}\]
elements and the set of rank $k$ triangles consists of
\[l_k=\frac{N-4}{\sqrt{N-6}}(\Phi_1^{2k-3}+\Phi_2^{2k-3})+(N-3)(\Phi_1^{2k-2}+\Phi_2^{2k-2})\]
elements, where
$$\Phi_{1,2}=\frac{\sqrt{N-6}\pm\sqrt{N-2}}{2}$$
are the two roots of the equation $$\Phi^2-\sqrt{N-6}\Phi-1=0.$$ The action of $T$ on the set of rank $k$ $N$-gons is a cyclic permutation $i\mapsto i+p_k$ where 
\[p_k=\frac{q_k}{3}+\frac{\Phi_1^{2k-4}-\Phi_2^{2k-4}}{\sqrt{(N-6)(N-2)}}+\frac{\Phi_1^{2k-3}-\Phi_2^{2k-3}}{\sqrt{N-2}},\]
and the action of $T$ on the set of rank $k$ triangles is also a cyclic permutation $i\mapsto i+j_k$ where
\[j_k= \frac{l_k}{3}+(N-4)\frac{\Phi_1^{2k-4}-\Phi_2^{2k-4}}{\sqrt{(N-6)(N-2)}}+(N-3)\frac{\Phi_1^{2k-3}-\Phi_2^{2k-3}}{\sqrt{N-2}}.\]
The rotation number of the outer billiard map at infinity is given by the formula 
\[\rho(t)=\lim_{k\to\infty}\frac{p_k}{q_k}=\lim_{k\to\infty}\frac{j_k}{l_k}=\frac{1}{3}+\frac{1}{3(1+\Phi_1^2)}=\frac{1}{3}+\frac{1}{3\sqrt{N-2}\Phi_1}.\]
\end{theorem}

Theorem 2 contains many independent results  and for reasons of clarity we will prove them one by one as claims inside the proof.

\begin{claim}
Every orbit of $T$ is periodic.
\end{claim}
\begin{proof}
The proof of this  result is written in much detail in \cite{dogtab}. We will present here a sketch of it and will refer the reader to \cite{dogtab} for detailed explanations.
The statement of the claim is a consequence of the following lemma:
\begin{lemma} \label{rankpreserve} The rank of a tile is preserved under $T$.
\end{lemma}
\noindent\textit{Proof of lemma.} The proof is  by  induction  on the rank, based on geometrical observations. Observe that rank 1 tiles are preserved by $T$ and notice that every rank $k$ tile is adjacent to a rank $k-1$ tile, where these two tiles map together under a single application of $T$. These two facts complete the base case and the step of the induction. \QED

From Lemma \ref{rankpreserve}, since there are finitely many tiles of rank $k$, every tile must eventually map back to itself after $m$-iterations, for some natural number $m$. Hence the $m$-th iteration of $T$ maps the entire  tile to itself. This  implies that $T^{\circ m}$ is a rotation by either $\frac{2\pi j}{N}$ (for $N$-gons) or $\frac{2\pi j}{3}$(for triangles) around some point inside the tile. Hence $T^{\circ Nm}$ restricted to that tile is the identity if the tile is an $N$-gon and $T^{\circ 3m}$ restricted to that tile is the identity if the tile is a triangle. We conclude that every orbit of $T$ is periodic.
\end{proof}    

\begin{claim}
For every $k\geq 1$, $T$ permutes the rank $k$ tiles cyclically.
\end{claim}
\begin{proof}
This claim is an immediate corollary to the following lemma:
\begin{lemma}
Any two consecutive rank $k$ tiles are mapped to two consecutive rank $k$ tiles.
\end{lemma}
\noindent\textit{Proof of lemma.} We  know by Lemma \ref{rankpreserve} that the rank of two tiles is preserved. If the two consecutive tiles are not separated by a clockwise continuation of one of the sides of the table then their common point is mapped, together with the two tiles, through the same vertex. Thus the tiles are mapped to two consecutive tiles. 

If the two tiles are separated by such a continuation of one side of the table then  the argument is more involved. A similar argument is presented  in \cite{dogtab}. Figure 3 gives a pictorial representation of the situation. The first tile is reflected in $O_1$, while the second one is reflected in $O_2$. What remains to prove is that $A^{\prime}$=$B^{\prime}$ so that the images of the two tiles still touch in one point. The following sequence of equalities completes the proof:
\[ A^{\prime} O_2=A{^\prime}O_1-O_1O_2=BO_1+AB-O_1O_2=BO_1+O_1O_2=BO_2=B^{\prime}O_2.\]

\begin{figure}
\includegraphics[scale=0.7]{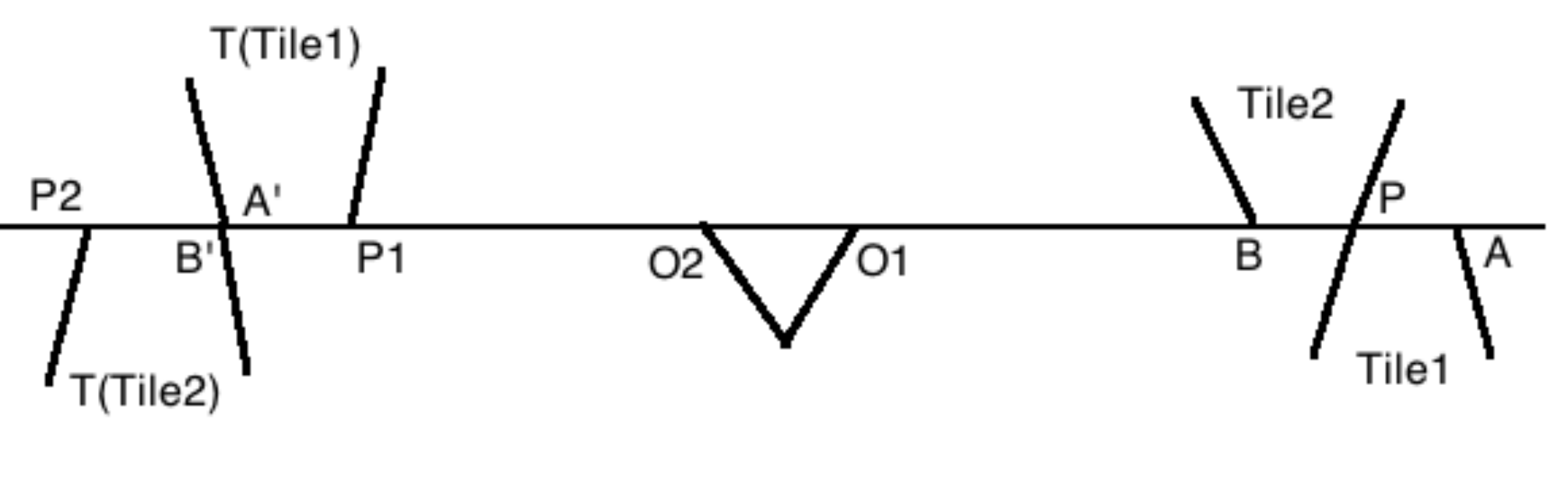}
\caption{Special case for lemma 2}
\end{figure}
\end{proof}

In order to  compute the formulas for $q_k,p_k,j_k,l_k$ we first explain why the tiling we are working with has an intrinsic self-similar geometric structure. We will refer from now on to this self-similar structure as the \textit{crochet pattern}. To describe the crochet pattern, we consider $N$-gons to be of two types (See figure \ref{37tiling}), $X$-type and $Y$-type. Type $X$ $N$-gons have two \lq\lq parents"  in the sense that they touch two $N$-gons of the previous rank, while type $Y$ $N$-gons touch only one  \lq\lq parent". The rank 1 $N$-gons are of neither of the types, having 0 parents, so we call them type \lq 0' $N$-gons. (This is why our counting argument begins with counting rank $2$ $N$-gons.) 

The following claim gives an intuitive explanation of why we  call this self-similar structure of the tiling a crochet pattern.
\begin{claim}
When passing from the $k$-th layer of $N$-gons to the $k+1$-th layer of $N$-gons, we apply the following replacement rules:
\[X\to XY^{N-6} \]
\[ Y\to  XY^{N-5} \]
i.e., when incrementing rank of the layer by 1, every $X$ gets replaced by an $X$ followed by $N-6$ $Y$'s, and every $Y$ gets replaced by an $X$ followed by $N-5$ $Y$'s.
\end{claim}

\begin{proof}
The methods used to prove this claim have been developed by Poincar\'{e} and we will not dwell on the details here. The reader can find extensive explanation in \textit{The Symmetry of Things} \cite{symm}.

Instead, we will illustrate the methods used to prove the claim in the case of $N=7$ in order to give the geometrical intuition behind the proof. Figure \ref{37tiling} illustrates the local and global behavior of a $(3,7)$-tiling. 

In the local picture, the difference between a type $X$ $7$-gon and a type $Y$ $7$-gon is encoded in the different types of degenerate heptagons we associate to them. We associate to the $Y$-type heptagon a rectangle with 3 additional points  on the upper side, while to the $X$-type heptagon we associate a rectangle with 2 additional points on the upper side and 1 additional point on the lower side, since it has two parents. Now by reducing the triangles in the global picture to points, we notice that the heptagons must meet 3 in each vertex. This results in the crochet pattern shown in Figure 4 (left). This crochet pattern immediately implies the claimed replacement rules.  
\end{proof}
\begin{figure}
\includegraphics[scale=0.6]{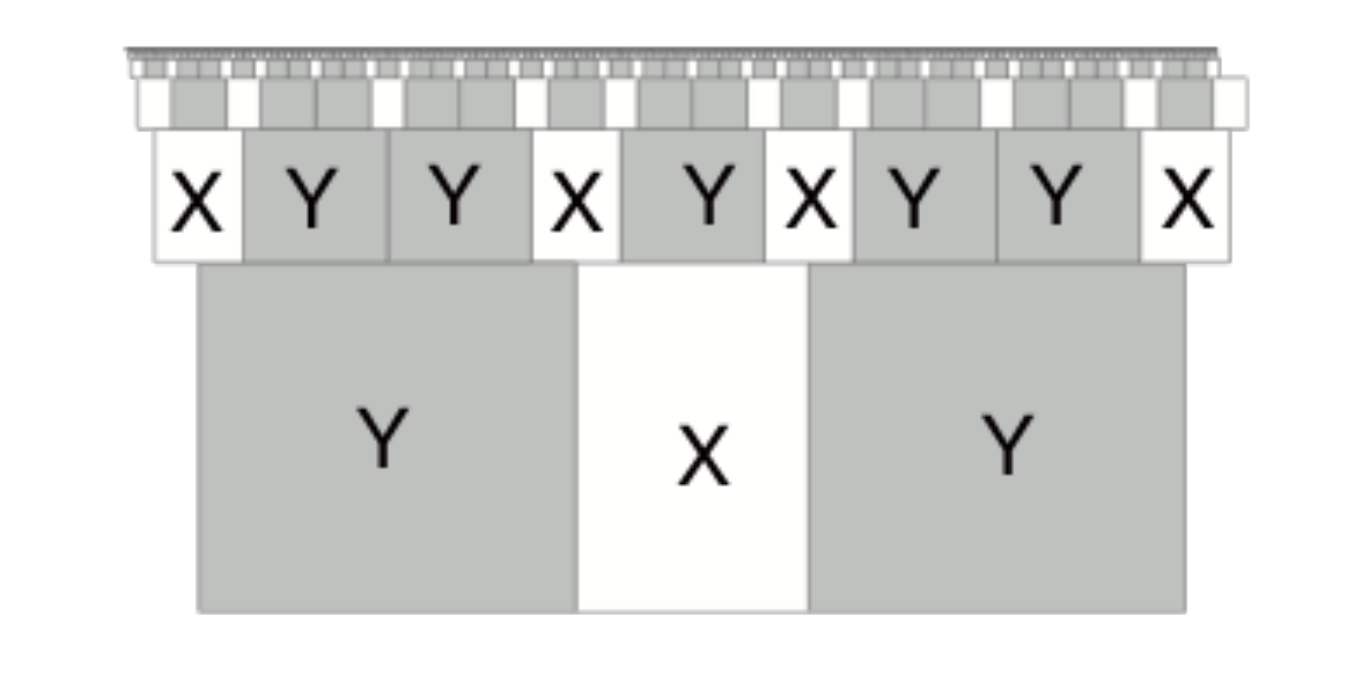}
\includegraphics[scale=0.3]{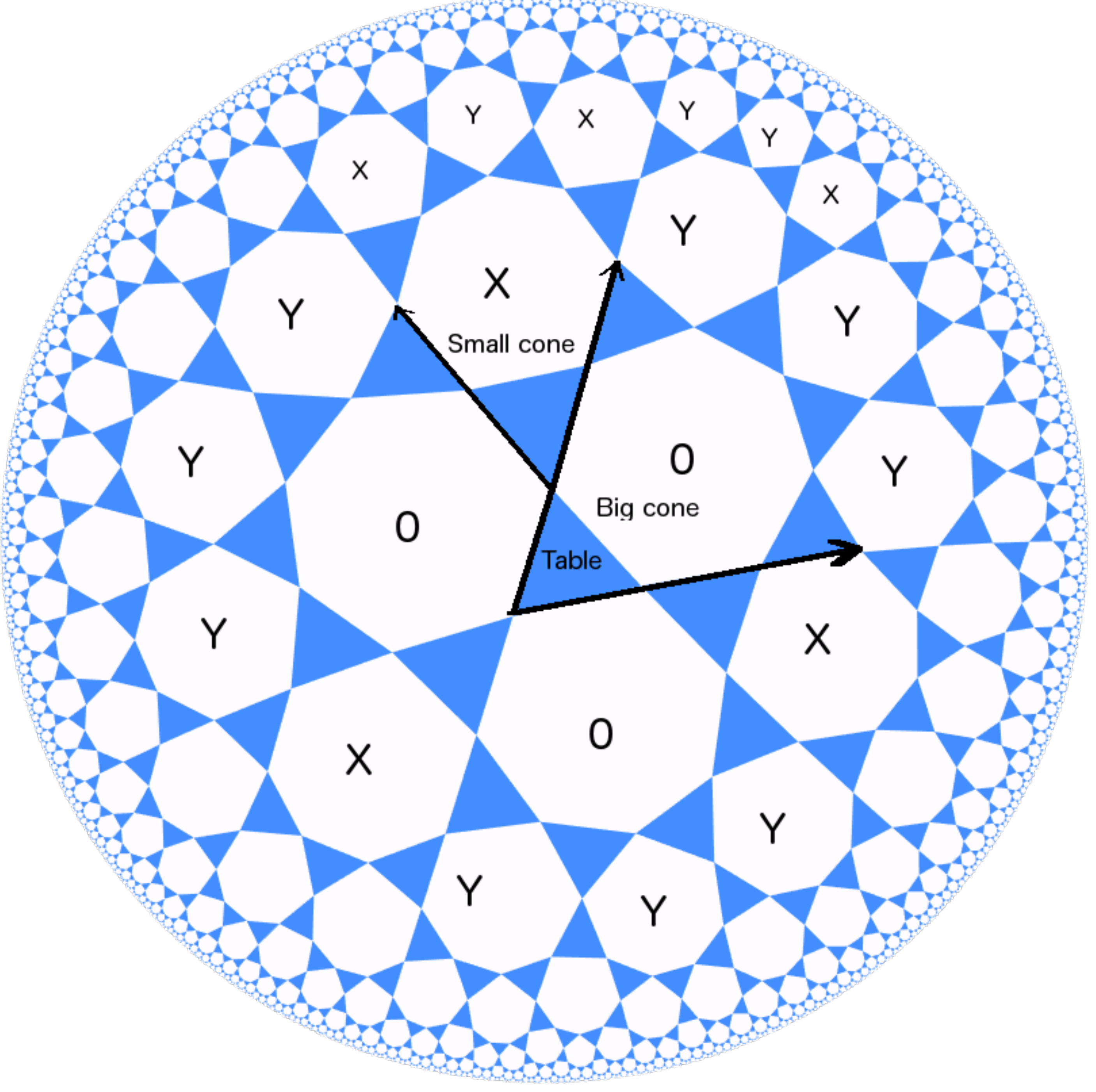}
\caption{The (3,7)-tiling}
\label{37tiling}
\end{figure}

We can now use this crochet pattern to start our counting argument in order to get the exact numbers in Theorem \ref{3n}. 
\begin{claim}
The formulas for $q_k,p_k,j_k,l_k$ hold as stated in Theorem \ref{3n}.
\end{claim}
\begin{proof}

Denote the number of $X$-type and $Y$-type $N$-gons of rank $k$ by $x_k$ and $y_k$, respectively, and use Claim 3 to obtain the following system of linear difference equations:
\[\left(\begin{array}{c}
x_k \\
y_k
\end{array}\right)=\left(\begin{array}{cc}
1&1 \\
N-6 & N-5
\end{array}\right)\left(\begin{array}{c}
x_{k-1} \\
y_{k-1}
\end{array}\right).\] 

The initial configuration is $\left(\begin{array}{c}
x_2 \\
y_2
\end{array}\right)=\left(\begin{array}{c}
3 \\
3(N-4)
\end{array}\right)$, because there must be three rank 2 $N$-gons with two parents, and the rest of the vertices of the rank 1 $N$-gons must serve as an anchor for a different $Y$-type rank 2 $N$-gon. Solving this recurrence gives the general term formula:

\[\left(\begin{array}{c}
x_k \\
y_k
\end{array}\right)=3\left(\begin{array}{c}
\frac{1}{\sqrt{N-6}}(\Phi_1^{2k-3}+\Phi_2^{2k-3}) \\
\Phi_1^{2k-2}+\Phi_2^{2k-2}
\end{array}\right)\]
where 
$$\Phi_1=\frac{\sqrt{N-2}+\sqrt{N-6}}{2}\ \ {\rm and}\ \ \Phi_2=\frac{-\sqrt{N-2}+\sqrt{N-6}}{2}.$$
 From here the formula for $q_k=x_k+y_k$ follows immediately. 

To count the triangles of rank $k$, we observe that the triangles of rank $k$ are the next layer after the $N$-gons of rank $k$, and each $X$-type $N$-gon is replaced by $N-4$ triangles and each $Y$-type is replaced by $N-3$ triangles. Hence the formula for $l_k=(N-4) x_k+ (N-3) y_k$ can be computed.

In order to count how many rank $k$ $N$-gons $T$ jumps, i.e., $p_k$, we need to define $s_k$ as the number of rank $k$ $N$-gons in a small cone as can be seen in Figure \ref{37tiling}. A small cone is opposite to one of the triangles vertices and doesn't contain any side of the triangle. In the same way, a big cone (see Figure \ref{37tiling}) is opposite to one of the sides of a triangle and contains the table. The number of rank $k$ $N$-gons in a big cone is just $\frac{q_k}{3}-s_k$ because of the 3-fold symmetry of the tiling.

For the same reasons as above we need to introduce $x^s_k$ and $y^s_k$, the number of $X$-type and $Y$-type rank $k$ $N$-gons in a small cone, respectively. With this, $s_k=x^s_k+y^s_k$. The billiard map $T$ makes any tile jump over 2 small cones and one big cone so in total it will jump $$p_k=2s_k+(\frac{q_k}{3}-s_k)=\frac{q_k}{3}+s_k.$$
By  studying the structure small cone we  observe the crochet pattern once again. One can notice that the cone that starts at the last $X$-type $N$-gon of the rank $k$ ($k\geq 2$) layer looks exactly the same as the initial small cone. That is why $s_k$ is equal to the total number of $N$-gons obtained by starting with an $X$-type $N$-gon and using the replacement rules in Claim 3. We express this as a sum:
\[\left(\begin{array}{c}
x^s_k \\
y^s_k
\end{array}\right)=\sum_{i=0}^{k-2}\left(\begin{array}{cc}
1&1 \\
N-6 & N-5
\end{array}\right)^i\left(\begin{array}{c}
1 \\
0
\end{array}\right),\] 
which, after some computation becomes:
\[\left(\begin{array}{c}
x^s_k \\
y^s_k
\end{array}\right)=\left(\begin{array}{c}
1+\frac{\Phi_1^{2k-4}-\Phi_2^{2k-4}}{\sqrt{(N-6)(N-2)}} \\
-1+\frac{\Phi_1^{2k-3}-\Phi_2^{2k-3}}{\sqrt{N-2}}
\end{array}\right).\] 

The formula for $p_k=\frac{q_k}{3}+x^s_k+y^s_k$ follows immediately. $j_k$ is computed in the same manner as $l_k$ was computed. As we have already said, every $X$ type $N$-gon is replaced by $N-4$ triangles and every $Y$ type $N$-gon is replaced by $N-3$ triangles on the next level and this procedure leaves uncounted only one rank $k$ triangle in the small cone, so $j_k=(N-4)x^s_k+(N-3)y^s_k+1$.
\end{proof}
\begin{claim}
The rotation number $\rho(t)$ equals
$$\frac{1}{3}+\frac{1}{3(1+\Phi_1^2)}=\frac{1}{3}+\frac{1}{3\sqrt{N-2}\Phi_1}.$$
\end{claim}
\begin{proof}
The $k$-th layer of $N$-gons gives a discrete approximation of the circle map at infinity and so $\frac{p_k}{q_k}$ is an approximation of $\rho(t)$ as $k$ goes to $\infty$. By taking the limit we obtained the desired formula for the rotation number $\rho(t)$.
\end{proof}

This last claim completes the proof of all the statements in Theorem 2. 

\begin{remark}
{\rm 
(i) One might expect the formulas in Theorem 2 to also work for $N=6$, i.e., a $(3,6)$-tiling of the Euclidean plain. That is not the case even though the crochet pattern works exactly the same also in the $(3,6)$-tiling. The difference that appears when computing the formulas in the $(3,6)$-tiling is that the matrix of the difference system is not diagonalizable anymore and so its powers look completely different. 

(ii) Note that the determinant of all the matrices given by the crochet pattern is 1. We believe this is true because the crochet pattern replacement can also be reversed, i.e., starting with rank $k$ layer, we can construct the rank $k-1$ layer.

(iii) According to Theorem 2, one can express the eigenvalues $\Phi_1$ and $\Phi_2=1/\Phi_1$ via the rotation number $\rho(t)$. Therefore this rotation number determines the numbers $q_k,l_k,p_k,j_k$, and hence the whole dynamics of the map $T$. 
}
\end{remark}

\subsubsection{General $(M,N)$-tilings}
\indent \indent Next we consider the case of a general $(M,N)$-tiling. The theorem and subsequent proof are analogous to those in the $(3,N)$ case in the previous subsection, but we must consider the cases separately due to a difference in the counting method. In the previous section, $N$-gons were classified into types $X$ and $Y$, having two parents and one parent, respectively. However, due to the difference in geometry of triangles versus generic $M$-gons, the tilings in the $M\ge4$ case never produce $N$-gons with two parents. In this case, $N$-gons either have one parent or no parent, which we denote as types $Y$ and $Z$. This alternate counting method will be explained in detail in the proof, but first we state the theorem:

\begin{theorem}
For an $(M,N)$-tiling with
$$M,N\geq 4,\ \frac{1}{M} + \frac{1}{N} < \frac{1}{2},$$
the outer billiard map $T$ preserves the rank of a tile and every orbit of $T$ is periodic. 
The set of rank $k$ $N$-gons consists of 
$$
q_k = \frac{M}{\sqrt{b^2-4}}\left( 
(b+1)(\alpha_1^{2k-2} - \alpha_2^{2k-2}) 
 - (\alpha_1^{2k-4}-\alpha_2^{2k-4})
\right)
$$
elements, and the set of rank $k$ $M$-gons consists of
$$
l_k = \frac{M(N-2)}{\sqrt{b^2-4}}
\left(
b(\alpha_1^{2k-2} - \alpha_2^{2k-2}) 
- (\alpha_1^{2k-4}-\alpha_2^{2k-4} )
\right)
$$
elements, where
$b = (M-2)(N-2)-2$ and
$$\alpha_{1,2} = \frac{\sqrt{b-2} \pm \sqrt{b+2}}{2}$$
are the two roots of the equation $\alpha^2-\sqrt{b-2}\alpha-1=0$. The action of $T$ on the set of rank $k$ $N$-gons is a cyclic permutation $i\mapsto i+p_k$ where 
\[p_k=\frac{q_k}{M} + \frac{M-2}{(b-2)\sqrt{b+2}}\left((b-1)(\alpha_1^{2k-3}-\alpha_2^{2k-3})-(\alpha_1^{2k-5}-\alpha_2^{2k-5})\right),
\]
and the action of $T$ on the set of rank $k$ $M$-gons is also a cyclic permutation $i\mapsto i+j_k$ where 
\[j_k= \frac{l_k}{M} + \frac{1}{(b-2)\sqrt{b+2}}\left(
 (b^2-2)(\alpha_1^{2k-3}-\alpha_2^{2k-3})-b(\alpha_1^{2k-5} - \alpha_2^{2k-5}) \right).
\]
The rotation number of the outer billiard map at infinity is given by the formula 
\[\rho(t)=\frac{1}{M} + \frac{M-2}{M\sqrt{b-2}\alpha_1}
\left( \frac{(b-1)\alpha_1^2-1}{(b+1)\alpha_1^2-1} \right). \]
\end{theorem}

\begin{remark}
{\rm
if $N=M$, the statement of Theorem 3 reduces to that of Theorem 1.
}
\end{remark}

The proof of Theorem 3 also consists of several steps.

\begin{claim} Every orbit of $T$ is periodic.
\end{claim}
\begin{proof}
The proof of this claim is analogous to the proof in the previous section. Because the rank of each tile is preserved under the billiard map, and because there are finitely many tiles of a given rank, every tile must map back to itself after some finite number of iterations $m$. When the tile maps back to itself, it has rotated by $\frac{2\pi j}{M}$ if it is an $M$-gon or by $\frac{2\phi j}{N}$ if it is an $N$-gon. Then $T^{\circ mM}$ is the identity if the tile is an $M$-gon and $T^{\circ mN}$ is the identity if the tile is an $N$-gon. 
\end{proof}
\begin{claim} For every $k\geq 1$, $T$ permutes the rank $k$ tiles cyclically. \end{claim}
\begin{proof} Proof is similar to that for Claim 2. \end{proof}

\begin{figure}[h]
\centering
\includegraphics[width=.60\textwidth]{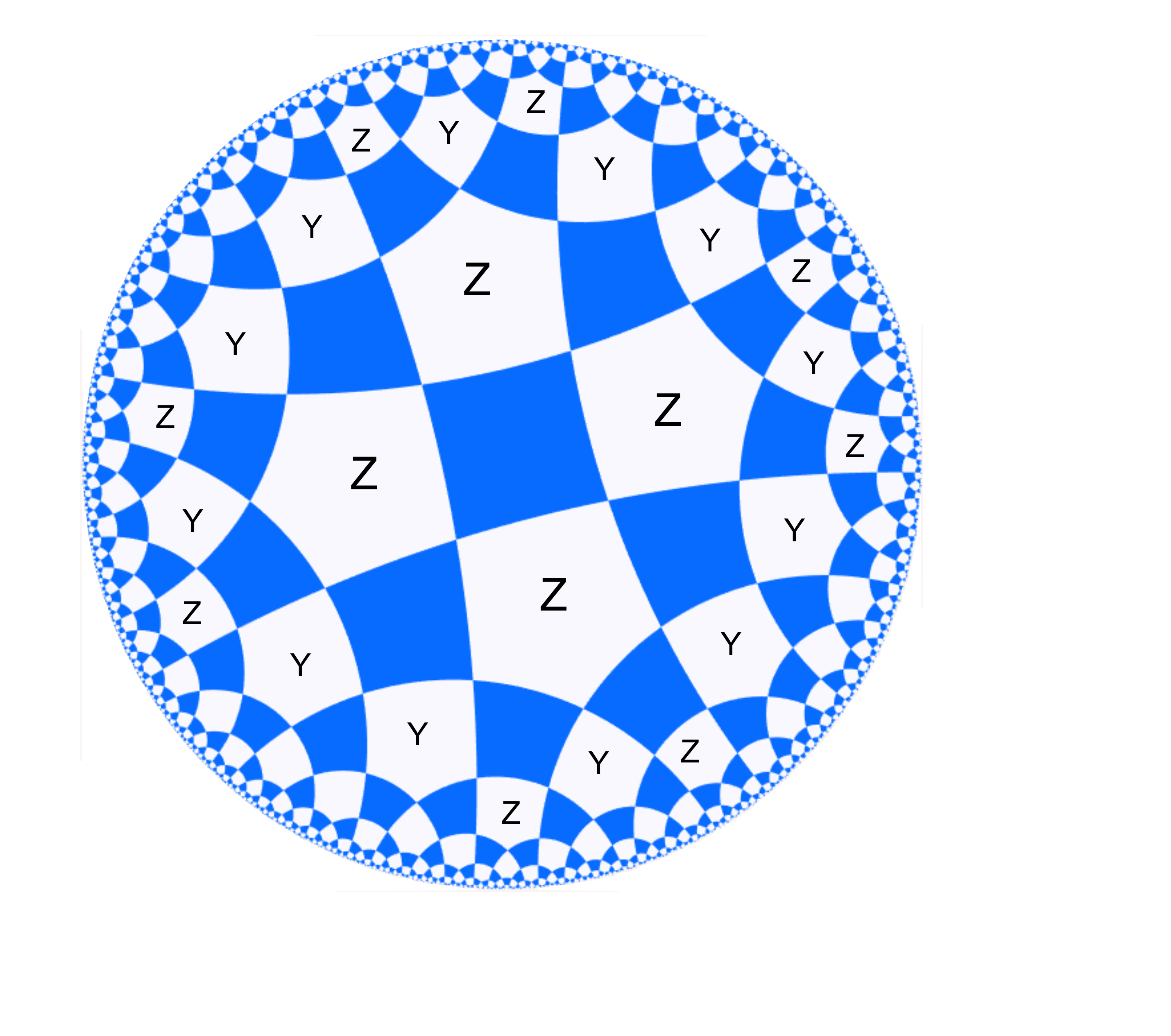}
\caption{(4,5)-tiling, with rank 1 and rank 2 pentagons labeled either as type Y (one parent) or as type Z (no parents).}
\label{fig:4-5labeled}
\end{figure}

Recall that we defined type $Y$ tiles to have one parent and type $Z$ tiles to have zero parents (see Figure \ref{fig:4-5labeled}). We now give a crochet pattern for general $(M,N)$-tilings, $M\geq 4$.

\begin{figure}[b]
\centering
\includegraphics[width=.5\textwidth]{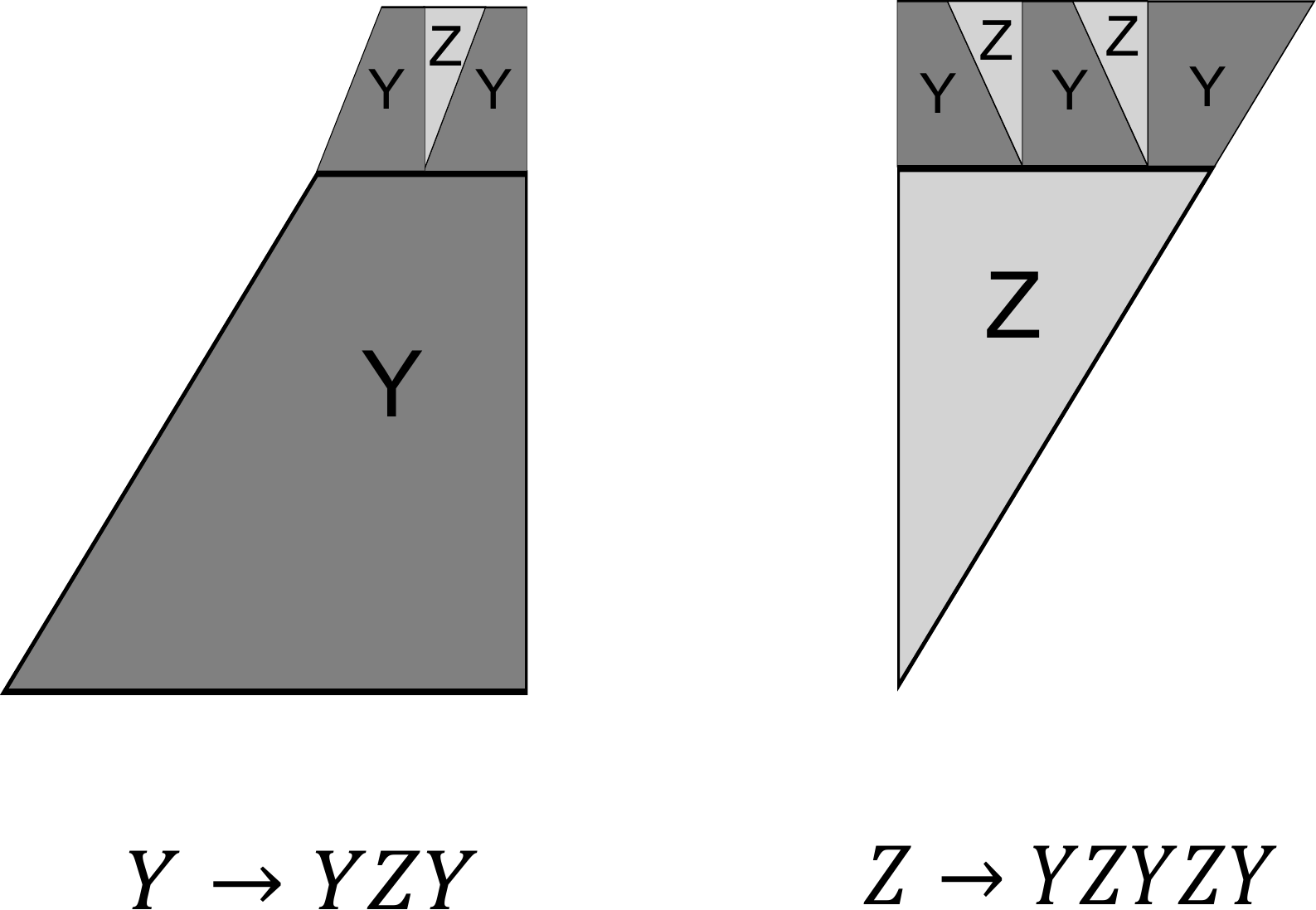}
\caption{Tiling of pentagons meeting in fours. Can be extended to a (4,5)-tiling.}
\label{fig:4-5crochet}
\end{figure}

\begin{claim} The following replacement rules hold for $(M,N)$-tilings.
\begin{equation} \label{eq:crochety} Y \to (YZ^{M-3})^{N-4}YZ^{M-4} \end{equation}
\begin{equation} \label{eq:crochetz} Z \to (YZ^{M-3})^{N-3}YZ^{M-4} \end{equation}
\end{claim}

\begin{figure}[h]
\centering
\includegraphics[width=.80\textwidth]{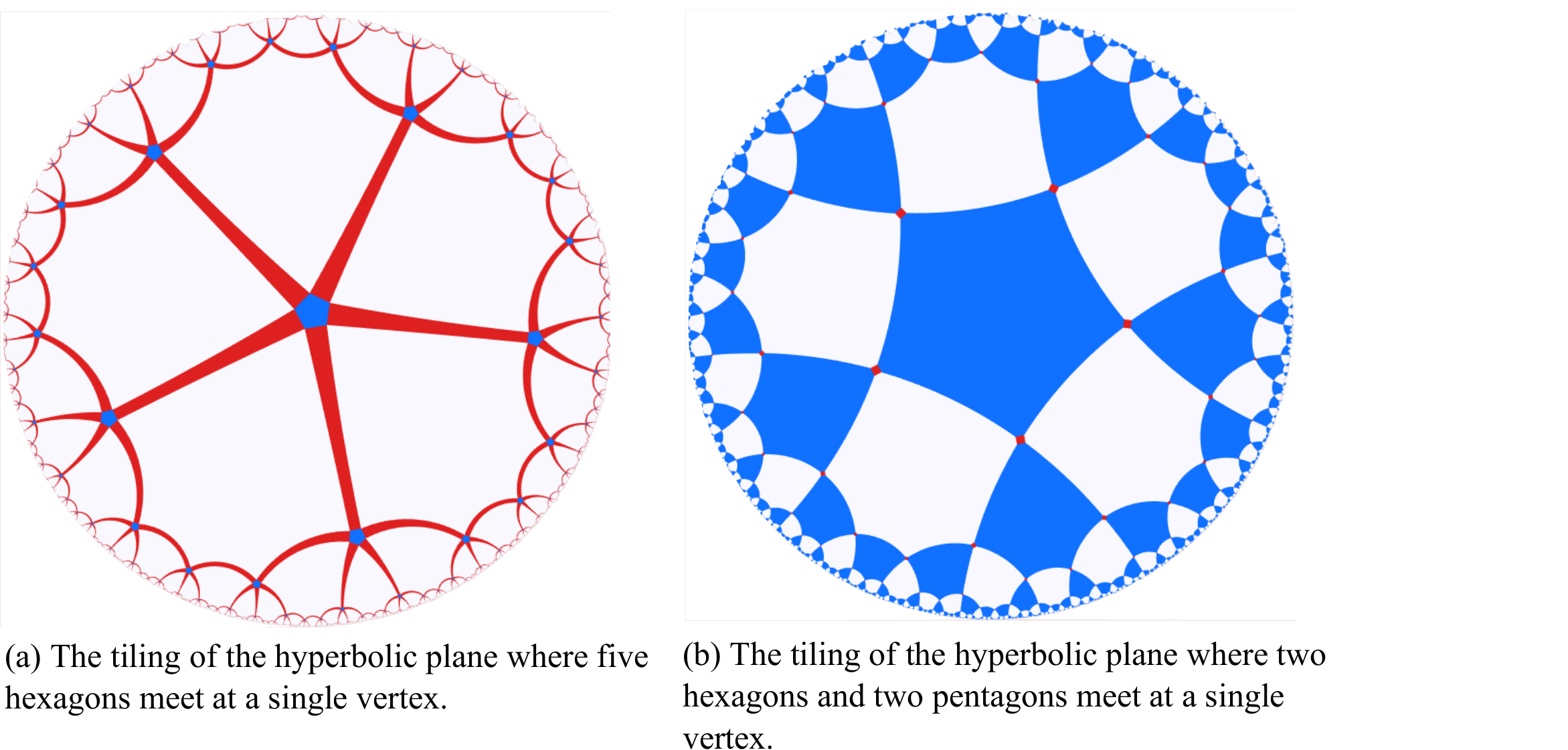}
\caption{In (a), the plane is tiled using hexagons meeting in fives. As illustrated by (b), by replacing the vertices in the previous picture with pentagons, we achieve a (5,6)-tiling.}
\label{fig:kaleido}
\end{figure}

\begin{proof}
In a similar manner to the $(3,N)$ case, we represent type $Y$ and $Z$ tiles as degenerate polygons, with additional vertices. See Figure \ref{fig:4-5crochet} for illustrations of the $(4,5)$ case. Type $Y$ tiles are represented as quadrilaterals with $N$ vertices, and type $Z$ tiles are represented as triangles with $N$ vertices. Because a $Y$ tile has $N-3$ sides available to connect with a tile of higher rank, a rank $k$ $Y$ tile produces $N-3$ $Y$ tiles of rank $k+1$. Then, since tiles must meet $M$-to-a-vertex, there must be $M-3$ $Z$ tiles between every pair of $Y$ tiles, and there must be $M-4$ type $Z$ tiles following the last $Y$. Similarly, a $Z$ tile has $N-2$ edges free to connect to a tile of higher rank, so a rank $k$ $Z$ tile produces $N-2$ $Y$ tiles of rank $k+1$, again with $Z$ tiles appropriately interspersed.

This crochet pattern tiles the hyperbolic plane with $M$ $N$-gons meeting at every vertex. From this tiling, we obtain the $(M,N)$-tiling by considering the points in the tiling becoming $M$-gons, as in Figure \ref{fig:kaleido} (compare with \cite{symm}). The described crochet pattern translates to the replacement rules given above.
\end{proof}

We can now compute the formulas for the number of $M$- and $N$-gons of any rank, as well as for the cyclic permutation of $M$- and $N$-gons of any rank.
\begin{claim}
The formulas for $q_k,p_k,j_k,l_k$ hold as stated in Theorem 3.
\end{claim}
\begin{proof}
Denoting the number of $Y$ type and $Z$ type $N$-gons of rank $k$ by $y_k$ and $z_k$, respectively, we obtain the following recursion formula:
\begin{equation} \label{eq:yz-recurrence}
\begin{pmatrix}
  y_{k} \\
  z_{k}
 \end{pmatrix}
=
A \begin{pmatrix}y_{k-1} \\ z_{k-1} \end{pmatrix}
\end{equation}
where the matrix $A$ is given below and is obtained from the rules given in (\ref{eq:crochety}) and (\ref{eq:crochetz}).
\begin{equation}
\label{eq:Amatrix}
A =
\begin{pmatrix}
N-3 & N-2 \\
(M-3)(N-3)-1 & (M-3)(N-2)-1
\end{pmatrix}.
\end{equation}
As mentioned above, the initial conditions are 
$\left(\begin{matrix} 
y_1 \\ z_1
\end{matrix}\right)
=
\left(\begin{matrix} 
0 \\ M
\end{matrix}\right)$. 

Solving the recurrence, we find the general formula:

$$
\begin{pmatrix}
  y_{k} \\
  z_{k}
 \end{pmatrix}
=
\begin{pmatrix}
\frac{M(N-2)(\alpha_1^{2k-2}-\alpha_2^{2k-2})}{\sqrt{b^2-4}} \\
\frac{M((M-3)(N-2)-1)(\alpha_1^{2k-2}-\alpha_2^{2k-2}) + M(\alpha_2^{2k-4}-\alpha_1^{2k-4})}{\sqrt{b^2-4}}
\end{pmatrix},
$$
where 
$$b = (M-2)(N-2)-2,\ 
\alpha_{1} = \frac{\sqrt{b-2} + \sqrt{b+2}}{2},\ 
\alpha_{2} = \frac{\sqrt{b-2} - \sqrt{b+2}}{2}.$$
Then $q_k = y_k + z_k$, so
$$
q_k = \frac{M}{\sqrt{b^2-4}}\left( 
(b+1)(\alpha_1^{2k-2} - \alpha_2^{2k-2}) +
\alpha_2^{2k-4} - \alpha_1^{2k-4}
\right).
$$

Now that we have counted the $N$-gons, we count the $M$-gons of rank $k$ by noticing a pattern in the tiling. 
We see that a type $Y$ $N$-gon of rank $k$ produces $N-3$ $M$-gons of rank $k$, and a type $Z$ $N$-gon produces $N-2$ $M$-gons.
Thus the number of $M$-gons of rank $k$ is given by $l_k = (N-3)y_k + (N-2)z_k$. 
The formula for $l_k$ given in Theorem 3 follows.

Next we determine $p_k$ by counting how many tiles a rank $k$ $N$-gon ``jumps" when $T$ is applied. 
As in the previous section, we define $s_k$ as the number of rank $k$ $N$-gons in a small cone.
We call $y_k^s$ and $z_k^s$ the number of rank $k$ $Y$s and $Z$s, respectively, in the small cone.
Also, as before, applying $T$ to any tile causes the tile to jump over two small cones and one big cone. 
In total, the jump is given by $p_k = s_k + \frac{q_k}{M}$.

We observe that
\begin{equation}
\label{eq:conesum}
\begin{pmatrix}
y_k^s \\ z_k^s
\end{pmatrix}
=
\sum_{i=0}^{k-2} A^i 
\begin{pmatrix}
2 \\ M-4
\end{pmatrix},
\end{equation}
where $A$ is given in (\ref{eq:Amatrix}).

This becomes
$$
\begin{pmatrix}
y_k^s \\ z_k^s
\end{pmatrix}
=
\begin{pmatrix}
\left(\frac{1-\alpha_1^{2k-2}}{1-\alpha_1^2} \right) + 
	\frac{2}{\sqrt{b^2-4}}\left( \frac{1-\alpha_2^{2k-2}}{1-\alpha_2^2}\right)(\alpha_1^2-N+3)
\\ 
\frac{1}{\sqrt{b^2-4}}\left[
\left(\frac{1-\alpha_1^{2k-2}}{1-\alpha_1^2} \right)
(B-\alpha_2^2(M-4))
+ 
\left( \frac{1-\alpha_2^{2k-2}}{1-\alpha_2^2}\right)
(-B+\alpha_1^2(M-4))
\right]
\end{pmatrix},
$$
where
$B = (M-3)(b-2)+(M-4)$.
Then, since $s_k = y_k^s + z_k^s$, we have
$$s_k = \frac{M-2}{(b-2)\sqrt{b+2}}
\left(
(b-1)(\alpha_1^{2k-3} + \alpha_2^{2k-3})
+ \alpha_2^{2k-5} - \alpha_1^{2k-5}
\right)
.$$

This allows us to calculate $p_k$ and we can compute $j_k$ by noticing again that every $Y$ type $N$-gon will be replaced by $N-3$ $M$-gons and every $Z$ type $(N-2)$-gon will be replaced by $N-3$ $M$-gons on the next level and this procedure will leave again only one $M$-gon out, so $j_k=(N-3)y^s_k+(N-2)z^s_k+1$.
\end{proof}

\begin{claim}
The rotation number is given by 
$$\rho(t)=\frac{1}{M} + \frac{M-2}{M\sqrt{b-2}\alpha_1}
\left( \frac{(b-1)\alpha_1^2-1}{(b+1)\alpha_1^2-1} \right).$$
\end{claim}
\begin{proof} This results from taking the limit of $\frac{p_k}{q_k}$ as $k \to \infty$.
\end{proof}

\section{Remarks and Acknowledgments}
\indent \indent The methods used in this paper both for geometrical and counting arguments can be used also for all other tilings with 2-fold symmetries in the vertices and so we believe that similar theorems and observations can be deduced in a more general setting. 

We want to acknowledge the work of two of our colleagues during the program. Stephanie Ger classified the tilings of the hyperbolic plane in terms of symmetry groups, and Ananya Uppal provided us with a Mathematica demonstration that created images that helped to understand the  patterns in the tilings. Their contributions were extremely valuable for our work.

We also want to thank our advisors during  the Summer@ICERM program, Chaim Goodman-Strauss, Sergei Tabachnikov, Ryan Greene and Tarik Aougab, for all the help and the inspiring discussions. Finally, we thank ICERM for providing the opportunity for us to participate in this great program. 

The tiling pictures were created using \textit{KaleidoTile} application created by Jeff Weeks \url{http://www.geometrygames.org/contact.html}.

\end{document}